\documentclass[11pt]{amsart}

\usepackage{amsmath, amsthm, amssymb}
\usepackage{amsfonts}
\usepackage{mathrsfs}
\usepackage{esint}
\usepackage[ansinew]{inputenc}
\usepackage[dvips]{epsfig}
\usepackage{graphicx}
\usepackage[english]{babel}
\usepackage{pdfsync}
\usepackage{hyperref}
\usepackage{color}
\usepackage{array}
\usepackage{appendix}
\usepackage{enumitem}

\usepackage{cleveref,autonum}

\usepackage[
  hmarginratio={1:1},     
  vmarginratio={1:1},     
  textwidth=17cm,        
  textheight=21cm,
  heightrounded,          
]{geometry}

\newtheorem{thm}{Theorem}[section]
\newtheorem{cor}[thm]{Corollary}
\newtheorem{lem}[thm]{Lemma}
\newtheorem{prop}[thm]{Proposition}

\newtheorem{defn}[thm]{Definition}
\newtheorem{rem}[thm]{Remark}

\newtheorem*{thm*}{Theorem}
\newtheorem*{defn*}{Definition}
\numberwithin{equation}{section}

\newcommand{\dx}{\,{\rm d}x}

\newcommand{\dt}{\,{\rm d}t}
\newcommand\dtau{{\,{\rm d}\tau}}
\newcommand{\rd}{{\rm d}}

\newcommand{\NN}{\mathbb{N}}

\newcommand{\RR}{\mathbb{R}}
\newcommand\EE{{\mathcal{E}}}

\newcommand\JJ{{\mathcal{J}}}
\newcommand\LL{{\mathcal{L}}}
\newcommand\UU{{\mathcal{U}}}

\newcommand{\vep}{\varepsilon}

\def\dist{\mathrm{dist}} 

\def\INT{\mathrm{int}} 
\newcommand{\loc}{\mathrm{loc}}

\newcommand{\leb}{\mathscr{L}}

\begin{document}

\title[]{Elliptic regularization of some semilinear parabolic free boundary problems}

\author[A. Audrito]{Alessandro Audrito}
\address{Alessandro Audrito \newline \indent
Politecnico di Torino, DISMA, Corso Duca degli Abruzzi 24, 10129, Torino, Italia.
 }
\email{alessandro.audrito@polito.it}

\author[T. Sanz-Perela]{Tom\'as Sanz-Perela}
\address{T. Sanz-Perela  \newline \indent
	Departamento de Matemáticas, Universidad Autónoma de Madrid, Ciudad Universitaria de Cantoblanco, 28049 Madrid, Spain
 }
\email{tomas.sanz@uam.es}

\date{\today} 


%
%
%
%
%
%
%
%
%
%
%

\thanks{This project has received funding from the European Union's Horizon 2020 research and innovation programme under the Marie Sk{\l}odowska--Curie grant agreement 892017 (LNLFB-Problems) and from the European Research Council (ERC) under the grant agreement 948029. The first author is supported by the INDAM-GNAMPA project number CUP-E53C22001930001.
	The second author is supported  by grants PID2020-113596GB-I00, PID2021-123903NB-I00, and RED2018-102650-T funded by MCIN/AEI/10.13039/501100011033 and by ``ERDF A way of making Europe''
}

\begin{abstract}
We prove existence of strong solutions to a family of some semilinear parabolic free boundary problems by means of elliptic regularization. Existence of solutions is obtained in two steps: we first show some uniform energy estimates and then we pass to the weak limit. To carry out this second step, we establish uniform non-degeneracy estimates for the approximating sequence as well as parabolic non-degeneracy and optimal regularity for the limit solution.
To the best of our knowledge, this is the first time the elliptic regularization approach is used in the context of parabolic obstacle problems.
\end{abstract}

\maketitle


%
%
%
%
\section{Introduction}
In this note we use \emph{elliptic regularization} to construct strong solutions to the following class of semilinear parabolic \emph{free boundary} problems 
\begin{equation}\label{eq:ParReacDiff}
\begin{cases}
\partial_t u - \Delta u = - f_\gamma(u) \quad &\text{ in }  Q := \mathbb{R}^n \times (0,\infty), \\
u|_{t=0} = u_0     \quad &\text{ in } \mathbb{R}^n,
\end{cases}
\end{equation}
where $n \geq 1$, $\gamma \in [1,2)$, $u_0 \geq 0$, and 
\begin{equation}\label{eq:Reaction}
f_\gamma(u) := \gamma \chi_{\{u > 0\}} u^{\gamma - 1}.
\end{equation}
Such parabolic free boundary problems appear in Chemical engineering (see \cite{Phillips1987:art}) and transport of thermal energy in plasma (see \cite{Martinson1976:art}). 
The mathematical study of its solutions and their free interfaces was initiated by Caffarelli in \cite{Caf78,Caf78Bis} (see also \cite{CF79}) in the case $\gamma = 1$, in which \eqref{eq:ParReacDiff} reduces to a version of the Stefan problem (see e.g. \cite{FriedKinder1975:art,CF79,FigRosSerra1,AudKuk2022:art}), and later extended in the elliptic setting to the range $\gamma \in (0,2)$ by Alt and Phillips \cite{AltPhilips1986:art}  (see also \cite{Phillips1983:art,Phillips1983bis:art}). 
In the parabolic setting, we mention the works of Weiss \cite{Weiss1999:art,Weiss2000:art} (see also \cite{ChoeWeiss2003}, by Choe and Weiss), where the classification of blow-ups and fine regularity properties of the free boundary were established, together with sharp bounds on the (parabolic) Hausdorff dimension of $\partial \{u > 0\}$. Notice that in \cite{Weiss1999:art} the singular range $\gamma \in [0,1)$ was considered too. 
To the best of our knowledge, there are not further results in the range $\gamma < 1$, except the very recent papers \cite{DeSilSav2022:art,DeSilSav2022bis:art}, \cite{DiKarVal2022:art} and \cite{DurGiac2022:art}, where the authors have studied the range $\gamma < 0$ in the elliptic framework.

Existence of weak/strong solutions to \eqref{eq:ParReacDiff} is well-known (see e.g.~\cite{FriedKinder1975:art}). 
However, the \emph{elliptic regularization} approach is, to the best of our knowledge, new in the context of parabolic obstacle problems, and presents some interesting features with respect to the classical approach. 
It is a \emph{variational} approximation procedure (also known as \emph{WIED} method), introduced in the works of Lions~\cite{Lions1965:art} and Oleinik~ \cite{Oleinik1964:art} (see also the article of De Giorgi~\cite{DeGiorgi1996:art}). 
Due to its intrinsic flexibility, elliptic regularization has been applied in many contexts, see for instance \cite{Ilmanen1994:art,MielkeStefanelli2011:art,BogeleinEtAl2014:art,AkagiStefanelli2016:art,BogeleinEtAl2017:art,RSSS:2019} in the parabolic setting, and also \cite{SerraTilli2012:art,SerraTilli2016:art} in the context of nonlinear wave equations (see also the more recent works \cite{AudritoSerraTilli2021:art,Audrito2021:art} concerning systems with strong competition and nonlocal parabolic problems, respectively).

The main idea of this approach is to approximate solutions to the \emph{parabolic} problem \eqref{eq:ParReacDiff} by using (absolute) minimizers of the functional
\begin{equation}\label{eq:Functional0}
\EE_\vep(w) := \int_0^\infty \frac{e^{- t/\varepsilon}}{\varepsilon} \bigg( \int_{\RR^n} (\vep|\partial_t w|^2 + |\nabla w|^2 ) \dx + 2\int_{\RR^n} w_+^\gamma  \dx \bigg) \dt,
\end{equation}
where $\vep \in (0,1)$ is a free parameter and $w_+$ is the standard notation for $\max\{w,0\}$. 
Indeed, it is not difficult to check that, when $\gamma \in (1,2)$, any critical point $u_\vep$ satisfies 
\begin{equation}\label{eq:PerProbEps}
\begin{cases}
-\varepsilon \partial_{tt} u_\vep + \partial_t u_\vep - \Delta u_\vep = - f_\gamma(u_\vep) \quad &\text{ in }  Q = \RR^n \times(0,\infty), \\
u|_{t=0} = u_0     \quad &\text{ in } \mathbb{R}^n,
\end{cases}
\end{equation}
in the weak sense (the case $\gamma = 1$ is more involved, see the comments below).
Problem~\eqref{eq:PerProbEps} is exactly \eqref{eq:ParReacDiff} ``up to'' the extra term $-\varepsilon \partial_{tt} u_\vep$, which makes the equation \emph{elliptic} for every $\vep > 0$, but \emph{degenerate} as $\vep \to 0$: it is thus reasonable to expect that, if minimizers of \eqref{eq:Functional0} enjoy some uniform boundedness properties, we may pass to the weak limit into \eqref{eq:PerProbEps} (along a suitable subsequence $\vep_j \to 0$) and obtain a solution $u$ to \eqref{eq:ParReacDiff}.

\smallskip

This plan has two main steps:

\begin{enumerate}[{label=\arabic{*}.}]
	\item Prove uniform energy estimates in the spirit of \cite{SerraTilli2012:art,SerraTilli2016:art} and deduce compactness of families of minimizers in suitable Sobolev spaces.
	
	\item Pass rigorously to the limit into \eqref{eq:PerProbEps} to obtain  \eqref{eq:ParReacDiff}.
\end{enumerate}

The main difficulties arising from the fact that we are dealing with a free boundary problem appear precisely in the second step and when $\gamma = 1$ (the case $\gamma \in (1,2)$ is quite standard since $f_\gamma$ is continuous).
Indeed, there are two issues that must be treated with care, and are the core of this paper.
First, we prove that minimizers $u_\varepsilon$ of $\EE_\vep$ satisfy \eqref{eq:PerProbEps}.
This is not obvious, as mentioned above, since the function $w \mapsto w_+$ is not differentiable at $w=0$; see \Cref{Lem:EulerEqMinFF} below.
Second, once we have the limit of the minimizers, denoted by $u$ (which is obtained by compactness and after passing to a subsequence), we show that we can take the limit $\varepsilon\to 0$ in \eqref{eq:PerProbEps} and obtain \eqref{eq:ParReacDiff}.
This requires to prove that $\chi_{\{u_\varepsilon>0\}} \to  \chi_{\{u>0\}}$ and this is also non-obvious.
Indeed, to prove it we establish a uniform non-degeneracy property for the family $u_\varepsilon$ close to free boundary points, as well as optimal regularity and parabolic non-degeneracy estimates for the limit $u$.
All these ingredients and fine estimates for the measure of $\partial \{u>0\}$ are crucially exploited to pass to the limit as $\vep \to 0$.
\smallskip

\medskip

It is important to stress that our approach is not only aimed to construct solutions in an alternative way with respect to the existing literature (regularizing $\chi_{\{ u > 0\}}$ in the nonlinearity $f_\gamma(u)$ and using classical tools, see~\cite{FriedKinder1975:art}).
Indeed, we expect some features of the elliptic problem (with $\varepsilon>0$) being inherited by solutions to \eqref{eq:ParReacDiff}, and we hope this could lead to a different approach to study parabolic free boundary problems in the future.
For this, the first step is to establish the appropriate convergence of approximating minimizers towards solutions to \eqref{eq:ParReacDiff}, as we do in this paper.

The techniques we use in this article allow us to prove existence of solutions for every $\gamma \in [1,2)$, and some of them do not work for $\gamma \in (0,1)$ (see \Cref{Rem:gammain01} below). 
The main reason is that we use in a crucial way the weak formulation of the problem, obtained by considering competitors of the form $u + \delta \varphi$ with $\delta >0 $ and $\varphi \in C^\infty_c(Q)$, and the weak convergence in $H^1$. 
Since for $\gamma\in (0,1)$ the function $u\mapsto f_\gamma(u)$ has a strong singularity at $u=0$, we cannot use the standard weak formulation, but we would need to consider competitors constructed through domain deformations, in order to avoid differentiating the function $u\mapsto u_+^\gamma$.
This approach is much more delicate and will be treated in a forthcoming article.
%
%
%
%
%
%
%

\begin{rem}
	As mentioned above, the elliptic regularization approach is quite flexible and more general/complex equations can be attacked with similar techniques (see for instance \cite{AkagiStefanelli2016:art,SerraTilli2016:art}). 
	Here we just mention that our methods can be slightly modified to prove existence of strong solutions to
	\begin{equation}\label{eq:BDDDomain}
	\begin{cases}
		\partial_t u - \Delta u = - f_\gamma(u) \quad &\text{ in }  \Omega \times (0,\infty), \\
		u|_{t=0} = u_0     \quad &\text{ in } \Omega,
	\end{cases}
	\end{equation}
	where $\Omega \subset \RR^n$ is a bounded domain and homogeneous Dirichlet or Neumann conditions are posed on $\partial \Omega \times (0,\infty)$. 
	The proof of such fact is postponed to Section \ref{Sec:Extension}; see Corollary \ref{cor:Dirichlet}.

\end{rem}

In the next subsection, we introduce the functional setting and we state our main result.
%
%
%
%
%
%
%
%
%
\subsection{Functional setting and main result}
We will work with the space
\[
\UU := \bigcap_{r > 0} H^1(Q_r^+), \qquad \text{ where } \ Q_r^+ := B_r \times (0,r^2),
\]
made of functions $u \in L^2(Q_r^+)$ with weak derivatives $\partial_t u \in L^2(Q_r^+)$, $\partial_i u \in L^2(Q_r^+)$ ($i = 1,\ldots,n$), for every $r > 0$. 
The functional \eqref{eq:Functional0} is well-defined on $\UU$ with values in $[0,+\infty]$: we will seek for minimizers $u \in \UU$, subject to the initial condition $u|_{t=0} = u_0$ in the sense of traces, where
\begin{equation}\label{eq:Ass2Data}
\ u_0 \in H^1(\RR^n) \cap L^\gamma(\RR^n) \quad \text{ and } \quad u_0 \geq 0 \quad \text{a.e. in } \RR^n,
\end{equation}
In other words, we will minimize $\EE_\vep$ over the space
\begin{equation}
\UU_0 := \{ u \in \UU: u|_{t=0} = u_0 \}.
\end{equation}
We remark that the assumptions on the initial data guarantee that $u_0$ (seen as a function of $x$ and $t$) belongs to $\UU_0$: in particular, $\UU_0 \not= \varnothing$. This will be useful in the proof of Lemma \ref{lem:ExistenceMin}, where we will use $u_0$ as a competitor and prove the elementary, yet crucial, estimate \eqref{eq:LevelEst}. We do not expect such assumptions to be optimal, but they have the advantage to make the arguments direct and easy-readable.

We now introduce the notion of strong solutions to \eqref{eq:ParReacDiff}.
\begin{defn}\label{def:StrongSolReacDiffProb}
Let $n \geq 1$ and $\gamma \in [1,2)$. 
We say that a function $u$ is a strong solution to \eqref{eq:ParReacDiff} if 

$\bullet$ $u \in L_\loc^2(0,\infty: H^1(\RR^n))$ with $\partial_t u \in L_\loc^2(0,\infty:L^2(\RR^n))$, and $u|_{t=0} = u_0$ in the sense of traces.

$\bullet$ The integral relation  
\begin{equation}\label{eq:EulerParabolicIntro}
\int_Q (\partial_t u \eta + \nabla u \cdot\nabla\eta + f_\gamma(u)\eta ) \dx\rd t = 0
\end{equation}
is satisfied for every $\eta \in C_c^\infty(Q)$.
\end{defn}

Note that this definition differs from one for weak solutions in the fact that  $\partial_t u$ is an $L^2$-function and the integral \eqref{eq:EulerParabolicIntro} involves the term $\partial_t u \eta $ and not $-u \partial_t \eta$.
In particular, the fact that $\partial_t u$ is in $L^2$ yields that $\Delta u$ is also in $L^2$.

Finally, we state our main result.
\begin{thm}\label{thm:MainIntro}
Let $n \geq 1$, $\gamma \in [1,2)$ , $u_0$ as in \eqref{eq:Ass2Data}, and $f_\gamma$ as in \eqref{eq:Reaction}. 
Then there exist a sequence of minimizers $\{u_{\vep_j}\}_{j\in\NN}$ of \eqref{eq:Functional0} in $\UU_0$ and a continuous strong solution $u \in \UU_0$ to \eqref{eq:ParReacDiff} such that 
\[
\begin{aligned}
&u_{\vep_j} \rightharpoonup u \quad \text{ weakly in } \UU \\
&u_{\vep_j} \to u \quad \text{ locally uniformly in } Q.
\end{aligned}
\]
\end{thm}
The rest of the paper is organized as follows. In \Cref{sec:GlobalUnifEnEst} we prove existence and uniqueness of minimizers of $\EE_\vep$ in $\UU_0$, and we establish the uniform energy estimates given in \eqref{eq:EnBound1} and \eqref{eq:EnBound2}. 
As a corollary, we prove weak and strong convergence of minimizers (see \Cref{prop:weaklimit} below).
\Cref{Sec:L2LinfEst} is divided into three parts.
First, in \Cref{Subsec:ELeq} we write the Euler-Lagrange equations for minimizers and show uniform non-degeneracy near free boundary points. 
Then, in \Cref{Subsec:EstimatesParabolic} we establish optimal regularity and parabolic non-degeneracy results for solutions to parabolic obstacle problems.
Finally, all these ingredients are combined crucially in \Cref{Subsec:ProofMainResult} to prove  \Cref{thm:MainIntro}. 
In the last section of the article, \Cref{Sec:Extension}, we comment on the slight modifications that one must perform in our methods in order to construct strong solutions to \eqref{eq:BDDDomain} with homogeneous Dirichlet or Neumann conditions on the parabolic boundary.


%
%
%
%
\section{Energy estimates and Compactness}\label{sec:GlobalUnifEnEst}
In this section we establish the following result.
\begin{prop}\label{prop:weaklimit} 
Let $\{u_\vep\}_{\vep \in (0,1)} \in \UU_0$ be a family of minimizers of $\EE_\vep$. Then there exist $u \in \UU_0$ and a sequence $\vep_j \to 0$ such that 
\begin{equation}\label{eq:WeakLimitProp}
\begin{aligned}
&u_{\vep_j} \rightharpoonup u \quad \text{ weakly in } \UU \\
&u_{\vep_j} \to u \quad \;\text{ in } L^2_{\loc}(Q).
\end{aligned}
\end{equation}
\end{prop}
Before addressing to the proof of the above statement, we show some basic properties of minimizers of the functional $\EE_\vep$.
%
%
\subsection{Existence of minimizers}
It is not difficult to show that for every $\vep \in (0,1)$ the functional $\EE_\vep$ has a unique minimizer in $\UU_0$ (see \Cref{lem:ExistenceMin} below). 
Before presenting the proof, we introduce the functional 
\begin{equation}\label{eq:ResFunctional}
\JJ_\vep(w) := \int_0^\infty e^{-t} \bigg( \int_{\mathbb{R}^n} \left( |\partial_t w|^2 + \vep|\nabla w|^2 \right)\rd x + 2\vep \int_{\mathbb{R}^n} w_+^\gamma \dx \bigg) \dt,
\end{equation}
which satisfies
\begin{equation}\label{eq:RelFuncFJ}
\EE_\vep(u) = \frac{1}{\varepsilon} \JJ_\vep(v), \quad \text{ whenever } v(x,t) = u(x,\varepsilon t).
\end{equation}
Since $v$ and $u$ coincide at $t=0$, minimizing $\EE_\vep$ in $\UU_0$ is equivalent to minimizing $\JJ_\vep$ in the same space. 
Working with the functional $\JJ_\vep$ allows us to keep the notations easier.
\begin{lem}\label{lem:ExistenceMin}
For every $\varepsilon \in (0,1)$, the functional $\JJ_\vep$ defined in \eqref{eq:ResFunctional} has a unique minimizer $v_\vep$ in $\UU_0$. 
Such minimizer satisfies $v_\vep \geq 0$ a.e. in $Q$. Moreover, there exists a constant $C > 0$, depending only on $n$, $\gamma$, and $u_0$, such that
\begin{equation}\label{eq:LevelEst}
\JJ_\vep(v_\vep) \leq C \varepsilon.
\end{equation}
\end{lem}
\begin{proof}
	By considering $u_0$ (as a function of $x$ and $t$) we see that, thanks to \eqref{eq:Ass2Data}, we have 
	\begin{equation}\label{eq:Energyu0}
		\begin{aligned}
			\JJ_\vep(u_0) & = \vep \int_0^\infty e^{-t} \bigg ( \int_{\mathbb{R}^n} |\nabla u_0|^2 \rd x +  2\varepsilon \int_{\mathbb{R}^n} (u_0)_+^\gamma \dx \bigg )\dt \\
			& \leq \varepsilon \big( \| u_0\|_{H^1(\RR^n)}^2  + 2 \|u_0\|_{L^\gamma(\RR^n)}^\gamma \big) \\
			& \leq C\varepsilon,
		\end{aligned}
	\end{equation}
	and thus $\JJ_\vep \not \equiv +\infty$ on $\UU_0$.
	To prove the existence of a minimizer, we consider a minimizing sequence $v_j \in \UU_0$:
\begin{equation}\label{eq:MinSeqCond}
\JJ_\vep(v_j) \to \inf_{v \in \UU_0} \JJ_\vep(v) := I_\vep \geq 0.
\end{equation}
Consequently, for every fixed $r > 0$, we have 
\[
\int_{Q_r^+} \left(|\partial_t v_j|^2 + \varepsilon |\nabla v_j|^2 \right)\rd x \dt \leq C,
\]
for some $C > 0$ independent of $j$. 
Since for every $j$ we have $v_j|_{t=0} = u_0$, given $\delta > 0$,  for a.e. $s \in (0, r^2)$ it holds
\[
\begin{aligned}
\|v_j(\cdot, s)\|_{L^2(B_r)}^2 &= \| u_0 \|_{L^2(B_r)}^2 + 2\int_0^s \int_{B_r} v_j(x,t)\partial_t v_j(x,t) \rd x \rd t\\
&\leq \| u_0 \|_{L^2(B_r)}^2 + \delta \int_{Q_r^+} |v_j|^2 \rd x \dt + \frac{1}{\delta} \int_{Q_r^+} |\partial_t v_j|^2 \rd x \dt.
\end{aligned}
\]
Integrating between $0$ and $r^2$ and taking $\delta := 1/(2r^2)$ we find 
\begin{equation}\label{eq:UnifL2Bound}
\int_{Q_r^+} v_j^2 \rd x \dt \leq 2 \left[ \|u_0\|_{L^2(B_r)}^2 + 2 r^2 \int_{Q_r^+} |\partial_t v_j|^2 \rd x \dt\right] \, r^2 \leq C,
\end{equation}
for some new $C > 0$ independent of $j$. Therefore, it follows that $\{v_j\}_j$ is bounded in $H^1(Q_r^+)$ uniformly in $j$ which, in turn, implies the existence of $v \in H^1(Q_r^+)$ such that $v_j \rightharpoonup v$ weakly in $H^1(Q_r^+)$ and $v_j \to v$ in $L^2(Q_r^+)$, up to passing to a subsequence. Now, since $\RR^n \times(0,\infty) = Q = \cup_{r>0} Q_r^+$, a diagonal argument shows that
\[
\begin{aligned}
&v_j \rightharpoonup v \quad \text{weakly in } \UU  \quad \text{ and } \quad v_j \to v  \quad \text{ in } L_{\loc}^2(Q),
\end{aligned}
\]
up to passing to another subsequence and, by the fact that $\UU_0$ is closed and convex, we conclude $v \in \UU_0$. Finally, by lower semicontinuity and Fatou's lemma, we obtain
\[
\JJ_\vep(v) \leq I_\vep,
\]
i.e., $v$ is a minimizer of $\JJ_\vep$ in $\UU_0$.

Now, uniqueness follows by the convexity of $\JJ_\vep$. 
Furthermore, since $u_0 \geq 0$, if $v_\vep$ is a minimizer then $(v_\vep)_+$ is an admissible competitor and thus, by minimality, $\JJ_\vep(v_\vep) \leq \JJ_\vep((v_\vep)_+)$ which is impossible unless $v_\vep \geq 0$ a.e. in $Q$.
Finally, \eqref{eq:LevelEst} follows from the minimality of $v_\vep$ and the bound~\eqref{eq:Energyu0}.
\end{proof}
%
%
%
%
%
\subsection{Proof of Proposition \ref{prop:weaklimit}}

\Cref{prop:weaklimit} will be obtained as a consequence of the following energy estimates.
\begin{prop}\label{thm:UnifEnergyEst} (Energy estimates)
There exists a constant $C > 0$, depending only on $n$, $\gamma$, and $u_0$, such that for every minimizer $u_\vep$ of $\EE_\vep$ we have
\begin{equation}\label{eq:EnBound1}
\int_0^\infty \int_{\mathbb{R}^n} |\partial_t u_\vep|^2 \dx \rd\tau \leq C
\end{equation}
and, for every $r \geq \varepsilon$,
\begin{equation}\label{eq:EnBound2}
\int_0^r \int_{\mathbb{R}^n} \big( |\nabla u_\vep|^2 + 2(u_\vep)_+^\gamma \big) \dx\rd\tau \leq Cr.
\end{equation}
\end{prop}

Indeed, with these uniform bounds at hand we readily obtain the main result of this section.

\begin{proof}[Proof of \Cref{prop:weaklimit}] 
	In light of \eqref{eq:EnBound1} and \eqref{eq:EnBound2}, the family $\{u_\vep\}_{\vep \in (0,1)}$ of minimizers of $\EE_\vep$ is uniformly bounded in $H^1(Q_r^+)$ for every fixed $r > 0$ and thus a standard diagonal procedure shows the existence of a sequence $\varepsilon_j \to 0$ and $u \in \UU_0$ such that the first limit in \eqref{eq:WeakLimitProp} is satisfied. 
	By the Sobolev embedding, we immediately obtain also the second limit, up to passing to another subsequence. 
\end{proof}

In the rest of the section we establish  \Cref{thm:UnifEnergyEst}.
This is the first key result of the paper and will be obtained as the byproduct of \Cref{Lemma:DerivEn} and  \Cref{cor:GlobalEnEstV}. 
In what follows we will use the following notations:
we will denote $\RR_+ := (0,+\infty)$ and, for a minimizer $v$ of $\JJ_\vep$, we write
\begin{equation}\label{eq:ShortJFuncInner}
\JJ_\vep(v) = \int_0^\infty e^{-t} ( I(t) + R(t) ) \, \rd t,
\end{equation}
where
\begin{equation}\label{eq:DefIandR}
I(t) := \int_{\mathbb{R}^n} |\partial_t v|^2\dx, \qquad R(t) := \varepsilon \int_{\mathbb{R}^n} \big( |\nabla v|^2 + 2v_+^\gamma \big) \dx.
\end{equation}
Since $v$ is a minimizer, it is clear that the functions $t \mapsto I(t)$, $t \mapsto R(t)$ and $t \mapsto e^{-t} ( I(t) + R(t) )$ are locally integrable in $\RR_+$. 
Consequently, the function
\begin{equation}\label{eq:Energy}
E(t) := e^t \int_t^\infty e^{-\tau} ( I(\tau) + R(\tau) ) \, \rd \tau,
\end{equation}
belongs to $W^{1,1}_{\loc}(\mathbb{R}_+)\cap C(\overline{\mathbb{R}_+})$, with $E(0) = \JJ_\vep(v)$. 

Notice that, by definition,
\begin{equation}\label{eq:DerivEn1}
E' = E - I - R \quad \text{ in } \mathcal{D}'(\mathbb{R}_+).
\end{equation}
In the following result we give an alternative expression for $E'$ which will be useful later to obtain energy estimates.
\begin{lem}\label{Lemma:DerivEn}
Let $v \in \UU_0$ be a minimizer of $\JJ_\vep$. Then
\begin{equation}\label{eq:DerivEn2}
E' = - 2I \quad \text{ in } \mathcal{D}'(\mathbb{R}_+).
\end{equation}
\end{lem}
\begin{proof} 
	We proceed in the spirit of \cite[Proposition 3.1]{SerraTilli2012:art}.
	Let $\eta \in C_c^\infty(\RR_+)$ and set $\zeta(t) := \int_0^t\eta(\tau)\rd\tau$.
	Given $\delta \in \mathbb{R}$, we define
	\begin{equation}\label{eq:DefPhi}
	\varphi(t) := t - \delta\zeta(t), \quad t \geq 0.
	\end{equation}
	If $|\delta| \leq \delta_0$ and $\delta_0 > 0$ is small enough, then $\varphi$ has smooth inverse $\psi := \varphi^{-1}$ given by
	\begin{equation}\label{eq:DefPsi}
	\psi(\tau) = \tau + \delta \zeta(\psi(\tau)).
	\end{equation}
	
	Now, from the minimizer $v \in \UU_0$, for $|\delta| \leq \delta_0$  we define the competitor
	\begin{equation}\label{eq:competitorwdelta}
	 w_\delta(x,t) := v(x,\varphi(t)).
	\end{equation}
	Since $\varphi(0) = 0$, we have $w_\delta|_{t=0} = u_0$ (which implies $w \in \UU_0$) and, by \eqref{eq:DefPhi}, $w_\delta |_{\delta = 0} = v$. 
	Changing variable $t = \psi(\tau)$ in \eqref{eq:ShortJFuncInner}, we easily see that
	\[
	\begin{aligned}
	\JJ_\vep(w) &= \int_0^\infty e^{-t} \bigg( \int_{\mathbb{R}^n} \left( |\partial_t w_\delta|^2 + \vep|\nabla w_\delta|^2 \right)\rd x + 2\vep \int_{\mathbb{R}^n} (w_\delta)_+^\gamma \dx \bigg) \dt \\
	&= \int_0^\infty e^{-t} \left[ \varphi'(t)^2 \, I(\varphi(t)) + R(\varphi(t)) \right] \rd t = \int_0^\infty \psi'(\tau) e^{-\psi(\tau)} \left[ \varphi'(\psi(\tau))^2 \, I(\tau) + R(\tau) \right] \rd \tau,
	\end{aligned}
	\]
	and thus, since $\varphi,\psi \in W^{1,\infty}(\RR_+)$, we deduce $\JJ_\vep(w_\delta) < +\infty$. 
	In particular, $w_\delta \in \UU_0$ is an admissible competitor for all $|\delta| \leq \delta_0$ and so, by minimality of $v$, 
	\begin{equation}\label{eq:GatDerJInner}
	\lim_{\delta \to 0^+} \frac{\JJ_\vep(w_\delta) - \JJ_\vep(v)}{\delta} = 0.
	\end{equation}
	Tedious, yet standard, computations using \eqref{eq:DefPhi} and \eqref{eq:DefPsi} show that 
	\begin{equation}\label{eq:CompInner}
	\frac{\rd}{\rd \delta} \left( \psi'(\tau) e^{-\psi(\tau)} \right) \Big|_{\delta = 0} =  \zeta'(\tau)e^{-\tau} - \zeta(\tau)e^{-\tau}, \qquad \frac{\rd}{\rd\delta} \left|\varphi'(\psi(\tau))\right|^2 \Big|_{\delta = 0} = -2\zeta'(\tau).
	\end{equation}
	Since $t \mapsto e^{-t} (I(t) + R(t))$ is locally integrable, we can pass to the limit in \eqref{eq:GatDerJInner} by dominated convergence and, in light of \eqref{eq:CompInner}, the limit in \eqref{eq:GatDerJInner} takes the form
	\begin{equation}\label{eq:MinCondInner}
	\int_0^\infty ( \zeta'(\tau)e^{-\tau} - \zeta(\tau)e^{-\tau} ) ( I(\tau) + R(\tau) ) \rd\tau - 2 \int_0^\infty e^{-\tau}\zeta'(\tau) \, I(\tau)  \dtau = 0.
	\end{equation}

	We are left to show the \eqref{eq:MinCondInner} is equivalent to \eqref{eq:DerivEn2}. 
	Recalling that $\zeta$ is a primitive of $\eta$ and testing the equation \eqref{eq:DerivEn1} with $\zeta'(\tau)e^{-\tau}$ and integrating by parts, we obtain 
	\begin{equation}\label{eq:AuxForInner1}
	\begin{aligned}
	\int_0^\infty \zeta'(\tau)e^{-\tau} ( I(\tau) + R(\tau) ) \rd\tau &= \int_0^\infty  E(\tau) \big[ \zeta'(\tau) e^{-\tau} + \big( \zeta'(\tau)e^{-\tau} \big)' \big] \rd\tau \\
	&= \int_0^\infty \zeta'(\tau) e^{-\tau} E(\tau) \dtau + \int_0^\infty E(\tau) \big( \eta(\tau)e^{-\tau} \big)' \rd\tau.
	\end{aligned}
	\end{equation}
	Using the definition of $E$ and integration by parts, the first term in the right-hand side of \eqref{eq:AuxForInner1} becomes
	\begin{equation}\label{eq:AuxForInner2}
	\begin{aligned}
	\int_0^\infty \zeta'(\tau) e^{-\tau} E(\tau) \dtau &= \int_0^\infty \zeta'(\tau) \int_\tau^\infty e^{-s} ( I(s) + R(s) ) \rd s \rd \tau \\
	&= \int_0^\infty \zeta(\tau) e^{-\tau} ( I(\tau) + R(\tau) ) \rd \tau.
	\end{aligned}
	\end{equation}
	Finally, combining \eqref{eq:AuxForInner1}, \eqref{eq:AuxForInner2}, and \eqref{eq:MinCondInner}, we deduce 
	\[
	\int_0^\infty E(\tau) \big( e^{-\tau} \eta(\tau) \big)' \rd\tau = 2 \int_0^\infty e^{-\tau} \eta(\tau) \, I(\tau)  \dtau,
	\]
	which, in turn, yields \eqref{eq:DerivEn2} thanks to the arbitrariness of $\eta \in C_c^\infty(\RR_+)$.
\end{proof}

From the previous result we obtain a corollary which will be the key to establish \Cref{thm:UnifEnergyEst}.
\begin{cor}\label{cor:GlobalEnEstV}
There exists a constant $C > 0$, depending only on $n$, $\gamma$, and $u_0$, such that for every minimizer $v_\vep$ of $\JJ_\vep$, we have
\begin{equation}\label{eq:BoundI}
\int_0^\infty \int_{\mathbb{R}^n} |\partial_t v_\vep|^2\dx \rd t \leq C\varepsilon,
\end{equation}
and, for every $r \geq 0$,
\begin{equation}\label{eq:BoundR}
\int_r^{r+1} \int_{\mathbb{R}^n} \big( |\nabla v_\vep|^2 + 2 (v_\vep)_+^\gamma \big) \dx\rd t \leq C.
\end{equation}
\end{cor}
\begin{proof} 
	First, since $I \geq 0$, the function $t \mapsto E(t)$ is non-increasing and so $E(t) \leq E(0)$ for all $t \geq 0$. 
	Recalling that $E(0) = \JJ_\vep(v_\vep)$ and the bound \eqref{eq:LevelEst} in \Cref{lem:ExistenceMin}, it follows that
	\begin{equation}\label{eq:EnergyBoundInner}
	E(t) \leq \JJ_\vep(v_\vep) \leq C\varepsilon,
	\end{equation}
	for all $t \geq 0$, where $C > 0$ is a constant depending only on $n$, $\gamma$, and $u_0$.

	Now, from \Cref{Lemma:DerivEn} we have $E' = -2I$ a.e. in $\RR_+$.
	Integrating this expression and using \eqref{eq:EnergyBoundInner}, we obtain
	\[
	2\int_0^t I(\tau)\dtau = E(0) - E(\tau) \leq E(0) \leq C\varepsilon,
	\]
	for all $t \geq 0$ (recall that $E \geq 0$ by definition) and \eqref{eq:BoundI} follows by the arbitrariness of $t \geq 0$ and the definition of $I$.

	To prove \eqref{eq:BoundR} for $r \geq 0$, it is enough to use \eqref{eq:EnergyBoundInner} to deduce that
	\[
	\varepsilon \int_r^{r+1} \int_{\mathbb{R}^n} \big( |\nabla v_\vep|^2 + 2(v_\vep)_+^\gamma \big) \dx \dt 
	=  \int_r^{r+1} R(t) \dt \leq e^{r+1} \int_r^{r+1} e^{-t} R(t) \dt \leq e E(r) \leq C\varepsilon.
	\]
\end{proof}
With the previous result at hand we can now establish the uniform bounds of \Cref{thm:UnifEnergyEst}.
\begin{proof}[Proof of \Cref{thm:UnifEnergyEst}] 
	Let $u:= u_\vep$ be a minimizer of $\EE_\vep$ in $\UU_0$ and let $v(x,t) := u(x,\varepsilon t)$. 
	Then, as a first consequence, we notice that \eqref{eq:EnBound1} is equivalent to \eqref{eq:BoundI}, which has been established above.
	Second, changing variable $\tau = \vep t$ in \eqref{eq:BoundR} yields
	\[
	\int_{\varepsilon \rho}^{\varepsilon \rho + \varepsilon}\int_{\mathbb{R}^n} \big( |\nabla u|^2 + 2u_+^\gamma \big)\dx\rd t \leq C\varepsilon,
	\]
	for all $\rho \geq 0$. 
	If $r = \varepsilon$ then \eqref{eq:EnBound2} follows from the above estimate taking $\rho = 0$.
	Otherwise, if $r > \vep$ let $k := \lceil r/\varepsilon \rceil \geq 2 $ and for  $j=0,1,\ldots$ apply the above estimate with $\rho = \rho_j = j$.
	Summing over $j = 0,\dots,k-1$ we obtain
	\[
	\int_0^{k\varepsilon}\int_{\mathbb{R}^n} \big( |\nabla v|^2 + 2v_+^\gamma \big)\dx\rd\tau \leq Ck\varepsilon.
	\]
	Since $ r \leq k\varepsilon \leq 2(k-1)\varepsilon\leq 2 r$, \eqref{eq:EnBound2} follows.
\end{proof}

\begin{rem}
	\label{Rem:gammain01}
	Notice that all the proofs (and, consequently, the statements) of this section work for the full range $\gamma \in [0,1)$ (when $\gamma = 0$, we set $u_+^\gamma := \chi_{\{u > 0\}}$).
	In the next section we are forced to restrict ourselves to the range $\gamma \in [1,2)$ in order to derive the Euler-Lagrange equation when considering competitors of the form $u + \delta\varphi$ with $\delta >0$ and $\varphi \in C_c^\infty(Q)$.
\end{rem}

%
%
%
%
%

%
\section{Proof of the main theorem}\label{Sec:L2LinfEst}
This section is devoted to the proof of \Cref{thm:MainIntro}, which is split in three steps. We first write the Euler-Lagrange equations for minimizers $u_\vep$ of $\EE_\vep$ (see \Cref{Lem:EulerEqMinFF} below) and then, in \Cref{Lem:NonDeg}, we establish a uniform non-degeneracy property. 
Next, we establish parabolic non-degeneracy and optimal regularity results for solutions to parabolic obstacle problems.
Last, we combine all these ingredients to pass two the limit as $\vep \to 0$, along a suitable subsequence, establishing \Cref{thm:MainIntro}.
Recall that we are using the notation $Q:= \RR^n\times (0,+\infty)$ and $f_\gamma(u) := \gamma \chi_{\{u > 0\}} u^{\gamma - 1}$.
\subsection{Euler-Lagrange equation and uniform non-degeneracy}
\label{Subsec:ELeq} Let us start with the Euler-Lagrange equations of minimizers.
\begin{lem}\label{Lem:EulerEqMinFF}
Let $\vep \in (0,1)$ and let $u_\vep \in \UU_0$ be the minimizer of $\EE_\vep$ in $\UU_0$. 
Then
\begin{equation}\label{eq:EulerEqMinFF}
\varepsilon \int_Q \partial_t u_\vep \partial_t \eta \dx \rd t + \int_Q (\partial_t u_\vep \eta + \nabla u_\vep \cdot\nabla\eta + f_\gamma(u_\vep)\eta) \dx \rd t = 0
\end{equation}
for every $\eta \in C_c^\infty(Q)$. 
\end{lem}
\begin{proof} 
	Let $u := u_\vep$ be the minimizer of $\EE_\vep$ in $\UU_0$ and let $\varphi \in C_c^\infty(Q)$.
	We have
	\[
	\begin{aligned}
	\frac{\EE_\vep(u + \delta\varphi) - \EE_\vep(u)}{\delta} &= 2\int_0^\infty \frac{e^{-t/\varepsilon}}{\varepsilon} \int_{\mathbb{R}^n} \left( \varepsilon \partial_t u \partial_t\varphi + \nabla u \cdot\nabla\varphi + \frac{(u+\delta\varphi)_+^\gamma - u_+^\gamma}{\delta} \right)\rd x \rd t + O(\delta),
	\end{aligned}
	\]
	as $\delta \to 0$. We divide the remaining part of the proof depending whether $\gamma \in (1,2)$ or $\gamma = 1$.

	\medskip
	
	\noindent $\bullet$ Assume $\gamma \in (1,2)$. Then the function $u \mapsto u_+^\gamma$ is everywhere differentiable, and thus
	\[
	\frac{(u+\delta\varphi)_+^\gamma - u_+^\gamma}{\delta} \to \gamma u_+^{\gamma-1} \varphi \quad \text{ a.e. in } Q,
	\]
	as $\delta \to 0$. Consequently, taking the limit of the incremental quotient above and using the minimality of $u$, we obtain
	\[
	\int_0^\infty e^{-t/\varepsilon} \int_{\mathbb{R}^n} \left( \varepsilon \partial_t u \partial_t\varphi + \nabla u \cdot\nabla\varphi + \gamma u_+^{\gamma-1} \varphi \right)\rd x \rd t = 0.
	\]
	Choosing $\varphi = e^{t/\varepsilon} \eta$ and noticing that $\partial_t\varphi = e^{t/\varepsilon}\left(\tfrac{1}{\varepsilon} \eta + \partial_t\eta \right)$, \eqref{eq:EulerEqMinFF} easily follows.
	
	\medskip
	
	\noindent $\bullet$ Assume $\gamma = 1$.
	In this case, the function $u \mapsto u_+$ is not differentiable at $u=0$: we thus generalize the argument in \cite[Proposition 5.12]{FerRERosOt2022:book} to our degenerate-elliptic setting. 
	 By Lemma \ref{lem:ExistenceMin}, we know that $u \geq 0$ a.e. in $Q$. Consequently, one easily obtains that 
	\[
	\frac{(u+\delta\varphi)_+ - u_+}{\delta} \; \to \;  
	\chi_{\{ u > 0\}}\varphi + \chi_{\{ u = 0\}}\varphi_+ \quad \text{ a.e. in } Q, 
	\]
	as $\delta \to 0^+$.
	Since by minimality $\EE_\vep(u + \delta\varphi) - \EE_\vep(u) \geq 0$, we may choose $\varphi = e^{t/\varepsilon} \eta$ as before and pass to the limit as $\delta \to 0^+$ to find
	\begin{equation}\label{eq:EulerEqMinFFProof}
	\varepsilon \int_Q \partial_t u \partial_t \eta \dx \rd t + \int_Q (\partial_t u \eta + \nabla u \cdot\nabla\eta + \chi_{\{ u > 0\}}\eta + \chi_{\{ u = 0\}}\eta_+) \dx \rd t \geq 0,
	\end{equation}
	for every $\eta \in C_c^\infty(Q)$. 
	
	Now, if we define the differential operator $\LL_\vep := \vep \partial_{tt} + \Delta$, setting
	\[
	- \langle \LL_\vep u, \eta \rangle := \int_Q ( \vep \partial_t u \partial_t \eta + \nabla u \cdot\nabla\eta ) \dx \rd t,
	\]
	the previous inequality \eqref{eq:EulerEqMinFFProof} becomes
	\[
	- \langle \LL_\vep u, \eta \rangle + \int_Q (\partial_t u \eta + \chi_{\{ u > 0\}}\eta + \chi_{\{ u = 0\}}\eta_+) \dx \rd t \geq 0. 
	\]
	Consequently, 
	\begin{equation}\label{eq:LepsSigned}
	\langle \LL_\vep u, \eta \rangle  \leq
	\begin{cases}
	\int_Q  (\partial_t u + 1)\eta \dx \rd t                \quad &\text{ for all } \eta \in  C_c^\infty(Q) \text{ such that } \eta \geq 0, \\ 
	\int_Q (\partial_t u + \chi_{\{u > 0\}})\eta \dx \rd t  \quad &\text{ for all } \eta \in  C_c^\infty(Q) \text{ such that } \eta \leq 0,
	\end{cases}
	\end{equation}
	that is,
	\begin{equation}\label{eq:DoubleIneq}
	\chi_{\{u > 0\}} \leq \vep\partial_{tt}u + \Delta u - \partial_tu \leq 1
	\end{equation}
	in the weak sense.

	Now, given $r > 0$, we bound $\LL_\vep u$ in $L^2(Q_r^+)$ by duality: given $\eta \in C_c^\infty(Q_r^+)$, \eqref{eq:LepsSigned} yields
	\begin{equation}
		\begin{split}
			\langle \LL_\vep u, \eta \rangle &\leq \langle \LL_\vep u, \eta_+ \rangle + \langle \LL_\vep u, -\eta_- \rangle \leq \int_{Q_r^+} |\partial_t u + 1|\eta_+ \dx \rd t + \int_{Q_r^+} |\partial_t u + \chi_{\{u > 0\}}||\eta_-| \dx \rd t \\
			&\leq 2\|\partial_tu + 1\|_{L^2(Q_r^+)} \|\eta\|_{L^2(Q_r^+)} \leq C \|\eta\|_{L^2(Q_r^+)}, 
		\end{split}
	\end{equation}
	where we have used the estimate \eqref{eq:EnBound1} in the last inequality (notice that here the constant $C >0$ depends on $r$ but is independent of $u$). 
	Thus, by the arbitrariness of $r > 0$, it follows that
	\[
	\LL_\vep u = \vep \partial_{tt}u + \Delta u \in L_{\loc}^2(Q),
	\]
	and so, by $W^{2,2}$-estimates of Calder\'on-Zygmund type, we obtain $u \in W^{2,2}_{\loc}(Q)$ (see for instance \cite[Section 2]{FerRERosOt2022:book}). 
	Since $W^{2,2}$-regularity is enough to apply Rademarcher's theorem, we deduce that
	\[
	\partial_tu, \partial_{tt}u, \Delta u = 0 \quad \text{ a.e. in } \{u = 0\}. 
	\]
	Combining this with \eqref{eq:DoubleIneq}, we finally obtain
	\[
	-\vep\partial_{tt}u - \Delta u + \partial_tu = - \chi_{\{u > 0\}} \quad \text{ a.e. in } Q,
	\]
	which, in turn, implies \eqref{eq:EulerEqMinFF} for $\gamma = 1$.
\end{proof}
We now establish a uniform non-degeneracy property for solutions to \eqref{eq:EulerEqMinFF} when $\gamma = 1$. 
Notice that, even though non-degeneracy is quite standard in obstacle problems, it is not clear that such property can be made \emph{uniform} with respect to $\vep \in (0,1)$, which is exactly what we prove next.

\begin{lem}\label{Lem:NonDeg}
There exists a constant $c > 0$, depending only on $n$, such that for every $\vep \in (0,1)$, every $(x_0,t_0)\in Q$ such that $B_1(x_0,t_0) \subset Q$, every weak solution $u_\vep$ to
\begin{equation}\label{eq:EulerNonDeg}
-\varepsilon \partial_{tt} u_\vep + \partial_t u_\vep - \Delta u_\vep = - \chi_{\{u_\vep > 0\}} \quad \text{ in }  B_1(x_0,t_0) \subset Q,
\end{equation}
every $(z_\vep,\tau_\vep) \in \overline{\{u_\vep > 0\}} \cap B_{1/2}(x_0,t_0)$, and every $r \in (0,\tfrac{1}{2})$, we have
\begin{equation}\label{eq:NonDeg}
\sup_{B_r(z_\vep,\tau_\vep)} u_\vep \geq cr^2.
\end{equation}
\end{lem}
\begin{proof} 
	Since the equation \eqref{eq:EulerNonDeg} is invariant under translations and here the initial condition at $\{t=0\}$  plays no role, we may assume $(x_0,t_0) = (0,0)$. 
	We also drop the $\vep$-subindexes to make the proof easier to read, that is, we set $u := u_\vep$, and $(z,\tau) := (z_\vep,\tau_\vep) \in \overline{\{u > 0\}} \cap B_{1/2}$.
	
	Now, take $\{(z_k,\tau_k)\}_{k\in \NN} \subset \{u > 0\}$ such that $(z_k,\tau_k) \to (z,\tau)$ as $k \to +\infty$. For each $k \in \NN$, we define
	\[
	w(x,t) := u(x,t) - u(z_k,\tau_k) - c \big( |x - z_k|^2 + (t - \tau_k)^2 \big), \qquad \text{ with } \quad c := \tfrac{1}{2(n+2)}.
	\]
	Then, if $r \in (0,1/2)$ is arbitrarily fixed, it is easy to compute
	\begin{equation}
		\begin{split}
			-\varepsilon \partial_{tt} w + \partial_t w - \Delta w & = - 1 + 2c(n +\vep) - 2c(t-\tau_k) \\
			& \leq - 1 + 2c(n + 1) + 2c|t-\tau_k| \\
			&\leq - 1 + 2c(n + 1) + 2c \leq 0 \qquad \text{ in } \{u > 0\} \cap B_r(z_k,\tau_k),
		\end{split}
	\end{equation}
	in view of the definition of $c$ above. 
	Note that $u$ (and hence $w$) is smooth in $\{u > 0\} \cap B_r(z_k,\tau_k)$ by standard elliptic estimates.
	By definition, we also have
	\[
	w(z_k,\tau_k) = 0 \qquad \text{ and } \qquad w < 0 \quad \text{ in } \partial\{u > 0\} \cap B_r(z_k,\tau_k),
	\]
	and thus, by the maximum principle,\footnote{Note that by standard elliptic estimates $u,w \in C^{1,\alpha}_{x,t}(B_{1/2})$ for all $\alpha\in (0,1)$ and $\{u > 0\} \cap B_r(z_k,\tau_k)$ is an  open set.}
	\begin{equation}
		\begin{split}
			0 = w(z_k,\tau_k) &\leq \sup_{\partial (\{u > 0\} \cap B_r(z_k,\tau_k))} w = \sup_{ \{u > 0\} \cap \partial B_r(z_k,\tau_k)} w \\
			&= \sup_{ \{u > 0\} \cap \partial B_r(z_k,\tau_k)} u - u(z_k,\tau_k) - cr^2 \leq \sup_{B_r(z_k,\tau_k)} u - u(z_k,\tau_k) - cr^2, 
		\end{split}
	\end{equation}
	that is
	\[
	\sup_{B_r(z_k,\tau_k)} u \geq u(z_k,\tau_k) + cr^2.
	\]
	Passing to the limit as $k \to +\infty$, \eqref{eq:NonDeg} follows thanks to the continuity of $u$.
\end{proof}
\subsection{Non-degeneracy and optimal regularity: parabolic setting}\label{Subsec:EstimatesParabolic}

We present below two technical results that we will exploit in the proof of \Cref{thm:MainIntro}, in the case $\gamma = 1$. 
The first one, \Cref{lem:ParNonDeg}, is a non-degeneracy property ---it can be seen as the parabolic version of Lemma \ref{Lem:NonDeg}---, while the second one, \Cref{cor:OptimalReg}, is an optimal regularity estimate, obtained by combining interior estimates and an optimal growth property established below in \Cref{lem:Optimalgrowth}. 
The optimal regularity we prove coincides with the optimal regularity of solutions to the parabolic obstacle problem (see for instance~\cite{CPS}). 
The proof closely follows the classical one, but we work in a slightly more general framework, which is exactly what we need to carry out the limiting procedure in the proof of \Cref{thm:MainIntro}.

Here and in the rest of the paper, we use the usual notation for parabolic cylinders.
That is, we set $Q_r(x_0,t_0) := \{(x,t) \in \RR^n \times \RR : |x-x_0|< r, |t - t_0| < r^2\}$ and  $Q_r^-(x_0,t_0) := Q_r(x_0,t_0)\cap \{t < t_0\}$.
As customary, the center point $(x_0,t_0)$ is omitted when $(x_0,t_0) = (0,0)$.

\begin{lem}\label{lem:ParNonDeg}
	Let $u\in C(Q_1)$ be a nonnegative weak solution to
	\begin{equation}
		\partial_t u - \Delta u = -1 \quad \text{ in } \{u>0\} \cap Q_1.
	\end{equation}
	Then, there exists  a constant $c_\circ >0$ depending only on $n$ such that, for every $(x_0,t_0) \in \overline{\{u>0\}} \cap  Q_{1/2}$ and every $r \in (0, 1/2)$, it holds
	\begin{equation}
		\sup_{Q_r^-(x_0,t_0)} u \geq c_\circ r^2.
	\end{equation}
\end{lem}
\begin{proof}
	The argument follows the proof of Lemma \ref{Lem:NonDeg}. Given $(x_0,t_0) \in \overline{\{u>0\}} \cap Q_{1/2}$, let $\{(x_k,t_k)\}_{k\in \NN} \subset \{u > 0\}\cap Q_{1/2}$ be such that $(x_k,t_k) \to (x_0,t_0)$ as $k \to +\infty$. 
	For each $k \in \NN$, we define
	\[
	w_k(x,t) := u(x,t) - u(x_k,x_k) - c_\circ \big( |x - x_k|^2 + t_k - t \big), \qquad \text{ with } \quad c_\circ := \tfrac{1}{2n+1}.
	\]
	Then, in $\{u>0\} \cup Q_1$ it holds
	\begin{equation}
	\partial_t w_k - \Delta w_k = -1 + c_\circ(2n + 1) \leq 0.
	\end{equation}
	Hence, using the maximum principle and that $w_k < 0$ in $\partial \{u>0\}\cup Q_r^-$ we have
	\begin{equation}
	0 = w_k(x_k, t_k) \leq \sup_{Q_r^-(x_k,t_k) \cup \{u>0\}} w_k = \sup_{\partial_p (Q_r^-(x_k,t_k) \cap \{u>0\})} w_k \leq \sup_{\partial_p (Q_r^-(x_k,t_k) )} w_k,
	\end{equation}
	where $\partial_p \Omega$ denotes the parabolic boundary of a set $\Omega \subset \RR^{n+1}$.
	Since $w_k \leq u - u(x_k,t_k) - c_\circ r^2$ in $\partial_p (Q_r^-(x_k,t_k) )$, it holds that
	\begin{equation}
		\sup_{\partial_p (Q_r^-(x_k,t_k) )} w_k \leq \sup_{\partial_p (Q_r^-(x_k,t_k) )} u  - u(x_k,t_k) - c_\circ r^2\leq \sup_{Q_r^-(x_k,t_k)} u  - u(x_k,t_k) - c_\circ r^2,
	\end{equation}
	from which we obtain
	\begin{equation}
		\sup_{Q_r^-(x_k,t_k)} u \geq  u(x_k,t_k) + c_\circ r^2.
	\end{equation}
	We conclude by taking the limit as $k\to +\infty$ and using that $u(x_0,t_0)\geq 0$ (note that  $u(x_0,t_0)= 0$ when $(x_0,t_0) \in \partial \{u>0\}$).
\end{proof}
Next we establish an optimal growth bound and, as corollary, an optimal regularity estimate. The proof is quite standard, but we present it for completeness, following \cite[Lemma 5.3]{AudKuk2022:art} where it is done for the parabolic obstacle problem.
\begin{lem}\label{lem:Optimalgrowth}
Let $K > 0$ and let $f \in L^\infty(Q_1)$ with $f \leq 0$ a.e. in $Q_1$ and $\|f\|_{L^\infty(Q_1)} \leq K$. Let $u\in C(Q_1)$ be a nonnegative weak solution to
\begin{equation}\label{eq:HEF}
\partial_t u - \Delta u = f \quad \text{ in } Q_1
\end{equation}
with $\|u\|_{L^\infty(Q_1)} \leq K$. Then, there exists $C_0 > 0$ depending only on $n$ and $K$ such that 
\begin{equation}\label{eq:OptimalGrowth}
||u||_{L^\infty(Q_r(x_0,t_0))}\leq C_0 r^2,
\end{equation}
for every $(x_0,t_0) \in \{u=0\}\cap Q_{1/2}$ and every $r \in (0,\tfrac{1}{4})$.   
\end{lem}
\begin{proof}
	Given $(x_0,t_0) \in \{u=0\}\cap Q_{1/2}$,
	since $\|f\|_{L^\infty(Q_1)} \leq K$, we first notice that \cite[Lemma 5.1]{Weiss1999:art} (or, equivalently, \cite[Lemma 5.2]{AudKuk2022:art}) yields that for every $\delta\in (0,1)$ and $r \in (0,\tfrac{1}{4})$,
\begin{equation}\label{eq:HarnackIneqParbola}
\sup_{P_r^\delta(x_0,t_0)} u \leq C_0( u(x_0,t_0) + Kr^2 ) = C_0 Kr^2,
\end{equation}
where $C_0 > 0$ depends only on $n$ and $\delta$, and
\begin{equation}
P_r^\delta(x_0,t_0) := \{(x,t) \in Q_r^-(x_0,t_0): t - t_0 < -\delta|x-x_0|^2\}.
\end{equation}
In light of \eqref{eq:HarnackIneqParbola}, it is sufficient to bound $u$ in the set $Q_r(x_0,t_0) \setminus P_r^\delta(x_0,t_0)$ for some $\delta\in (0,1)$ that will be chosen later. 
This will be obtained by comparison with the function 
\[
w(x,t) := a (t-t_0 + b|x-x_0|^2),
\]
where
\[
b := \tfrac{1}{2n}, \qquad a := 2 \cdot \max\{C_0,16\} \cdot \tfrac{K}{b}.
\]

Let us show that $u \leq w$ in $Q_{1/4}(x_0,t_0) \setminus P_{1/4}^\delta(x_0,t_0)$.
First, it is immediate to check that the choice of $b$ makes $w$ caloric in $\RR^{n+1}$. 
Now, on the one hand, setting $\delta := \tfrac{b}{2} \in (0,1)$, we have
\[
w(x,t) \geq  a (-\tfrac{\delta}{16}  + \tfrac{b}{16} ) = \tfrac{ab}{32}    \quad \text{ in } \partial_p Q_{1/4}(x_0,t_0) \setminus  \overline{ P_{1/4}^\delta(x_0,t_0)},
\]
which allows us to deduce that $u \leq w$ in $\partial_p Q_{1/4}(x_0,t_0) \setminus \overline{ P_{1/4}^\delta(x_0,t_0)}$, thanks to the assumption $\|u\|_{L^\infty(Q_1)} \leq K$ and the definition of $a$. On the other hand, for every $\rho \in (0,\tfrac{1}{4})$ we have
\begin{equation}\label{eq:OptGrowBdBelowSuper}
w(x,t)|_{t = t_0 -\delta|x-x_0|^2} = \tfrac{ab}{2} \rho^2 \quad \text{ in } \partial B_\rho(x_0),
\end{equation}
 by definition of $w$ and
\[
\sup_{x \in \partial B_\rho(x_0),t=t_0-\delta \rho^2} u(x,t) \leq C_0K \rho^2,
\]
for every $\rho \in (0,\tfrac{1}{4})$, as an immediate consequence of \eqref{eq:HarnackIneqParbola}. 
Combining the last two inequalities with the definition of $a$, we deduce that $u \leq w$ in $\partial P_{1/4}^\delta(x_0,t_0)  \cap \{t-t_0 > -\delta/16\}$ and therefore, since $u$ is sub-caloric in $Q_1$ (since $f \leq 0$ a.e. in $Q_1$), we obtain that $u \leq w$ in $Q_{1/4}(x_0,t_0) \setminus P_{1/4}^\delta(x_0,t_0)$ as wanted. 
The result then follows since $v \leq Cr^2$ in $Q_r(x_0,t_0) \setminus P_r^\delta(x_0,t_0)$ for every $r \in (0,\tfrac{1}{4})$ and some $C > 0$ depending on $a$ and $b$.    
\end{proof}
\begin{cor}\label{cor:OptimalReg}
Let $K > 0$, $f \in L^\infty(Q_1)$, and $u\in C(Q_1)$ as in Lemma \ref{lem:Optimalgrowth}. Further, assume that $f$ is constant in $\{u>0\}\cap Q_1$. 
Then there exists a constant $C > 0$ depending only on $n$ and $K$ such that 
\begin{equation}\label{eq:OptimalReg1}
\|\partial_t u\|_{L^\infty(Q_{1/2})} + \| D^2 u \|_{L^\infty(Q_{1/2})} \leq C.
\end{equation}
In addition, for every $(y,\tau) \in \{u > 0\}\cap Q_{1/2}$ we have
\begin{equation}\label{eq:OptimalReg2}
|\nabla u (y,\tau)| \leq C \delta,
\end{equation}
where $\delta := \sup \{ \rho > 0 : Q_\rho(y,\tau) \subset \{ u > 0 \} \}$.
\end{cor}
\begin{proof} 
	We may assume that $\partial \{u>0\} \cap Q_{1/2} \neq \varnothing$, otherwise the result is classical and well-known.
	Let us fix $(y,\tau) \in \{u > 0\}\cap Q_{1/2}$ and let $\delta = \delta(y,\tau) > 0$ be as in the statement. 
	We first apply the interior estimates in \cite[Theorem 4.9]{Lieberman1996:book} (with $\alpha = 1$, $(a_{ij})_{i,j=1}^n = I$, $b_i=0$ for $i=1,\ldots,n$, and $c=0$) to deduce 
\[
\begin{aligned}
\|\partial_t u\|_{L^\infty(Q_{\delta/2}(y,\tau))} + \| D^2 u \|_{L^\infty(Q_{\delta/2}(y,\tau))} &+ \frac{1}{\delta} \| \nabla u \|_{L^\infty(Q_{\delta/2}(y,\tau))} \\
&\leq \bar{C} \left( \frac{1}{\delta^2} \| u \|_{L^\infty(Q_\delta(y,\tau))} + \|f\|_{L^\infty(Q_\delta(y,\tau))} 
+ \delta [f]_{\text{Lip}_p(Q_\delta(y,\tau))} \right), 
\end{aligned}
\]
for some constant $\bar{C} > 0$ depending only on $n$, where
\[
[f]_{\text{Lip}_p(Q)} :=  \sup_{\substack{(x,t),(y,\tau) \in Q \\ (x,t)\not=(y,\tau)}} \frac{|f(x,t) - f(y,\tau)|}{\sqrt{|x-y|^2 + |t-\tau|}}. 
\]
Then, using the growth estimate \eqref{eq:OptimalGrowth} ---applied in $Q_{2\delta}(y_0,\tau_0)$---, the $L^\infty$ bound for $f$, and the fact that $f$ is constant on $Q_\delta$ ---since $Q_\delta\subset \{u>0\}$---, it follows that
%
\[
\frac{1}{\delta^2} \| u \|_{L^\infty(Q_\delta(y,\tau))} + \|f\|_{L^\infty(Q_\delta(y,\tau))} 
+ \delta [f]_{\text{Lip}_p(Q_\delta(y,\tau))}  \leq 4C_0 + K,
\]
where $C_0 > 0$ is as in \Cref{lem:Optimalgrowth}. Combining the above two inequalities and using the arbitrariness of $(y,\tau)$, we obtain
\begin{equation}\label{eq:OptRegEstInterior}
|\partial_t u (y,\tau)| + | D^2 u(y,\tau)| \leq C \qquad \text{ for all } \, (y,\tau) \in \{u > 0\}\cap Q_{1/2} 
\end{equation}
and 
\begin{equation}\label{eq:GradEstDelta}
|\nabla u(y,\tau)| \leq C \delta \qquad \text{ for all } \, (y,\tau) \in \{u > 0\}\cap Q_{1/2},
\end{equation}
where $C := \bar{C}(4C_0 + K)$. In particular, \eqref{eq:GradEstDelta} implies that $\nabla u$ can be continuously extended to zero in $\{u = 0\} \cap Q_{1/2}$. 

To complete the proof it is enough to show that $\nabla u$ is Lipschitz in space, that is
\begin{equation}\label{eq:GradLipSpace}
|\nabla u(x,t) - \nabla u(z,t)| \leq L|x-z| \qquad \text{ for all } \, (x,t),(z,t) \in Q_{1/2},
\end{equation}
for some constant $L > 0$ depending only on $n$ and $K$. 
Indeed, \eqref{eq:GradLipSpace} implies $\|D^2u\|_{L^\infty(Q_{1/2})} \leq L$ which, combined with the fact that $u$ is a weak solution to \eqref{eq:HEF} and $\|f\|_{L^\infty(Q_1)} \leq K$, gives the bound $\|\partial_t u\|_{L^\infty(Q_{1/2})} \leq L + K$.

In the next argument, we will use the \textit{parabolic distance}. 
For points $(x,t),(y,\tau) \in \RR^{n+1}$, it is defined as
\[
\dist_p((x,t),(y,\tau)) := \inf\{\rho > 0: (y,\tau) \in Q_\rho(x,t) \}.
\]
For a point $(x,t) \in \RR^{n+1}$ and a set $A \subset \RR^{n+1}$, we set $\dist_p((x,t),A) := \inf_{(y,\tau) \in A} \dist_p((x,t),(y,\tau))$.   

Let us establish \eqref{eq:GradLipSpace}.
To do it, let $(x,t),(z,t) \in Q_{1/2}$. 
If both $(x,t),(z,t) \in \{u = 0\}$ the claim is trivial since $\nabla u = 0$ in $\{u=0\}\cap Q_{1/2}$. 
Hence, we may assume that $(x,t) \in \{u > 0\}$.
By symmetry, we may also assume
\[
d := \dist_p((x,t),\{u = 0\}) \geq \dist_p((z,t),\{u = 0\}).
\]
On the one hand, assume that $\dist_p((x,t),(z,t)) \leq d/2$. Then, $(z,\tau) \in \overline{Q_{d/2}(x,t)} \subset \{u > 0\}$ and thus, by \eqref{eq:OptRegEstInterior}, we have
\[
|\nabla u(x,t) - \nabla u(z,t)| \leq \int_0^1 |D^2u(sx + (1-s)z,t)| \rd s \cdot |x-z| \leq C|x-z|. 
\]
On the other hand, if $\dist_p((x,t),(z,t)) \geq d/2$, then \eqref{eq:GradEstDelta} yields
\[
|\nabla u(x,t) - \nabla u(z,t)| \leq |\nabla u(x,t)| + |\nabla u(z,t)| \leq 2Cd \leq 4C \dist_p((x,t),(z,t)) = 4C |x-z|.
\]
Therefore \eqref{eq:GradLipSpace} follows with $L := 4C$. 
\end{proof}
\subsection{Proof of \Cref{thm:MainIntro}}\label{Subsec:ProofMainResult} In this subsection, we combine all the ingredients introduced above to prove our main result.
\begin{proof}[Proof of  \Cref{thm:MainIntro}]
	Let $\{u_\vep\}_{\vep \in (0,1)}$ be the family of minimizers of $\EE_\vep$ in $\UU_0$.
	Then by \Cref{prop:weaklimit}, there are $\vep_j \to 0$ and $u \in \UU_0$ such that $u_{\vep_j}$ converge to $u$ in the sense of \eqref{eq:WeakLimitProp} and, by \Cref{Lem:EulerEqMinFF}, each $u_{\vep_j}$ satisfies \eqref{eq:EulerEqMinFF} for every $j \in \NN$.
	For simplicity, we set $u_j := u_{\vep_j}$ for $j \in \NN$.
	
	\textbf{Step 1: local uniform convergence of the sequence of minimizers.} \\
	First, up to passing to another subsequence, we may assume that for every $r > 0$ and every $(x_0,t_0) \in Q$ such that $Q_r(x_0,t_0) \subset \subset Q$
	\[
	\|u_{j}\|_{C^{\alpha,\alpha/2}(Q_{r/2}(x_0,t_0))} \leq C r^{-\alpha} \big[ 1 + r^{-\frac{n+2}{2}} \|u_{j}\|_{L^2(Q_r(x_0,t_0))} \big]
	\]
	for some $\alpha \in (0,1)$ and $C > 0$ depending only on $n$ (the parabolic H\"older norm $\|\cdot\|_{C^{\alpha,\alpha/2}}$ is quite standard, see for instance \cite[Appendix C]{Audrito2021:art} for the definition). 
	This easily follows by combining the (re-scaled) estimates in \cite[Proposition 3.1 and Proposition 4.1]{Audrito2021:art} with $U_\varepsilon = u_{\varepsilon_j} = u_j$, $f_\varepsilon = 0$, $F_\varepsilon =- \chi_{\{u_{\varepsilon_j}	 >0\}}$, $p=q = \infty$, and $a=0$.\footnote{Notice that \cite[Proposition 4.1]{Audrito2021:art} can be applied here since $u_{\vep_j} \to u$ in $C_{\loc}(0,\infty:L^2(\RR^n))$ (this is an immediate consequence of \eqref{eq:EnBound1}-\eqref{eq:EnBound2} and \cite[Corollary 8]{Simon1987:art}) and thus the assumption (4.1) in \cite{Audrito2021:art} is satisfied.}
	Further, the sequence $\{u_j\}_{j\in \NN}$ is uniformly bounded in $L^2(Q_r^+)$ for every $r>0$ ---see the argument leading to \eqref{eq:UnifL2Bound}--- and thus
	\[
	\|u_j\|_{C^{\alpha,\alpha/2}(Q_{r/2}(x_0,t_0))} \leq C,
	\]
	for some other constant $C > 0$ independent of $j$. Therefore, combining the Arzelà-Ascoli theorem with a standard covering argument, and up to passing to another subsequence, we deduce that 
	\[
	u_{j} \to u \quad \text{ locally uniformly in } Q,
	\]
	which, in particular, implies that $u$ is continuous in $Q$.

	\textbf{Step 2: conclusion for $\gamma \in (1,2)$. } \\
	We show that $u$ is a strong solution to \eqref{eq:ParReacDiff}.
	Notice that, if $\gamma \in (1,2)$, the nonlinearity $f_\gamma$ is continuous. 
	Hence, we may pass to the limit in \eqref{eq:EulerEqMinFF} by dominated convergence, combining \eqref{eq:EnBound1} and \Cref{prop:weaklimit} with the continuity of $f_\gamma$, and deduce that $u$ satisfies \eqref{eq:EulerParabolicIntro}, i.e., \eqref{eq:ParReacDiff} in the sense of \Cref{def:StrongSolReacDiffProb}.

	\textbf{Step 3: limit problem for $\gamma = 1$.}\\
	In the rest of the proof, we assume $\gamma = 1$.
Since $\chi_{\{u_j > 0\}}$ are bounded in $L^\infty(Q)$ uniformly in $j\in \NN$, there exists a nonnegative function $\chi \in L^\infty(Q)$ such that $\chi_{\{u_j > 0\}} \rightharpoonup^\star \chi$ in $L^\infty(Q)$ up to passing to a subsequence. 
Consequently, passing to the limit in \eqref{eq:EulerEqMinFF} as above and recalling that $L^\infty(Q) = L^1(Q)^\star$, we have
\[
\int_Q (\partial_t u \eta + \nabla u \cdot\nabla\eta + \chi \eta ) \dx\rd t = 0,
\]
for every $\eta \in C_c^\infty(Q)$. Hence, to conclude, it is enough to show that
	\begin{equation}\label{eq:ChiAE}
		\chi = \chi_{\{u > 0 \}} \quad \text{  a.e. in } Q_R, 
	\end{equation}
where $Q_R$ is any cylinder such that $Q_{2R} \subset Q$ and $R > 0$.

To check \eqref{eq:ChiAE}, on the one hand we notice that if $(x,t)\in \{ u>0\} \cap Q_R$, then $\chi_{\{u_j>0\}}(x,t) = 1$ for $j$ large enough, by uniform convergence.
Therefore, $\chi = 1$ in $\{u>0\}\cap Q_R$. 
On the other hand, we claim that for any $\delta \in (0,R^2/4)$, if $(x,t) \in \{u = 0\} \cap Q_R$ with $\dist((x,t), \{u>0\})>\delta$, then 
\[
\chi_{\{u_j>0\}}(x,t) = 0
\]
for $j$ large enough. 
Indeed, if we assume that this is not true, there exist $\{(x_k,t_k) \}_{k\in \NN} \subset Q_R$ with $\dist((x_k,t_k), \{u>0\})>\delta$ and $(x_k,t_k) \in \{u_k>0\}$ for $k\in \NN$. Then, by our uniform non-degeneracy estimate \eqref{eq:NonDeg}, we have
\[
u_k(z_k, \tau_k)= \sup_{B_{\delta/2}(x_k,t_k)} u_k \geq c\delta^2,
\]
for some $(z_k, \tau_k) \in \overline{B_{\delta/2}(x_k,t_k)}$ and some $c > 0$ independent of $k$.
Up to passing to a subsequence we have that $(x_k, t_k) \to (x_0, t_0)$ and $(z_k, \tau_k)\to (z,\tau)\in \overline{B_{\delta/2}(x_0,t_0)} \subset \subset \{u>0\}^c$.
Thus, $u(x_0, t_0) = 0$, but this contradicts the fact that $u_k(z_k, \tau_k) \to u(x_0, t_0) \geq c\delta^2$ (which follows from uniform convergence). 
Consequently, the claim is proved and we conclude that  $\chi = 0$ in $\INT(\{u = 0\}) \cap Q_R$, where $\INT(A)$ denotes the interior of a set $A$.

Summing up, we have that $u\in C(Q_R)$ is a weak solution to $\partial_t u - \Delta u = - \chi$ in $Q_R$, where  $\chi \in L^\infty(Q)$ is nonnegative, $\chi = 1$ in $\{u > 0\}\cap Q_R$, and $\chi = 0$ in $\INT(\{u = 0\})\cap Q_R$ (in particular, $u$ fulfills the assumptions of both \Cref{lem:ParNonDeg} and \Cref{cor:OptimalReg} with $f = \chi$). 

To prove \eqref{eq:ChiAE}, it remains to show that $\partial\{u>0\}\cap Q_R$ has zero measure, which is what we do next, in the spirit of \cite[Theorem 5.1]{Weiss1999:art}. 

\textbf{Step 4: measure of the free boundary.}\\
We first prove that for every $(x,t) \in \partial\{u>0\}\cap Q_R$ and every $r \in (0,R/4)$, we have 
\begin{equation}\label{eq:Density}
	\dfrac{\leb^{n+1}(\{u>0\}\cap Q_r(x,t))}{\leb^{n+1}(Q_r(x,t))} \geq c_\star > 0,
\end{equation}
where $\leb^{n+1}$ denotes the $(n+1)$-dimensional Lebesgue measure and $c_\star$ is a constant depending only on $n$  and $K_R := \max \{1, \| u\|_{L^\infty(Q_R)}\}$.
To show \eqref{eq:Density}, we notice that \Cref{lem:ParNonDeg} yields the existence of $(y,\tau)\in \{u>0\}\cap \overline{Q_{r/2}(x,t)}$ such that
\begin{equation}\label{eq:NonDegFinal}
u(y,\tau) \geq \tfrac{c_\circ}{4} r^2,
\end{equation}
where $c_\circ > 0$ depends only on $n$. 
Now, we claim that $Q_{\tilde{c}r}(y,\tau) \subset \{u>0\} \cap  Q_r(x,t)$ if $\tilde{c} \in (0,\tfrac{1}{2})$ is small enough depending only on $n$  and $K_R$.
Indeed, given $(z,\theta) \in Q_{\tilde{c}r}(y,\tau)$,  setting $\gamma(s) := s(y,\tau) + (1-s)(z,\theta)$ for $s \in [0,1]$. 
By \Cref{cor:OptimalReg} (applied with $f= -\chi$ and $K = K_R$), we have
\[
\begin{aligned}
u(y,\tau) - u(z,\theta) &= \int_0^1 (\nabla u (\gamma(s)), \partial_tu (\gamma(s)) \cdot (y-z,\tau-\theta) \, \rd s \\
&\leq \sup_{s \in [0,1]} |\nabla u (\gamma(s))|\cdot|y-z| + \sup_{s \in [0,1]} |\partial_tu (\gamma(s))| \cdot |\tau-\theta| \leq Cr\cdot \tilde{c}r + C \cdot(\tilde{c}r)^2 \\
&\leq 2C\tilde{c}r^2,
\end{aligned}
\]
where $C > 0$ is the constant appearing in \eqref{eq:OptimalReg1} and \eqref{eq:OptimalReg2} and depends only on $n$ and $K_R$. 
Setting $\tilde{c} := \min\{ \tfrac{c_\circ}{16C},\tfrac{1}{2}\}$ and combining the bound above with \eqref{eq:NonDegFinal}, we obtain
\[
u(z,\theta) \geq (\tfrac{c_\circ}{4} - 2C\tilde{c}) r^2 > 0,
\]
and our claim is proved.
The fact that $Q_{\tilde{c}r}(y,\tau) \subset \{u>0\} \cap  Q_r(x,t)$ readily implies \eqref{eq:Density}.

Once \eqref{eq:Density} is established, let us show that the free boundary has zero measure.
By contradiction, we assume that $\leb^{n+1} (\partial\{u>0\}\cap Q_R) > 0$.
Now, since the set $\partial \{ u > 0\}$ is measurable, $\chi_{\partial \{ u > 0\}}$ is integrable in $Q_R$ and thus, for almost every point $(x,t)\in \partial \{ u > 0\} \cap Q_R$, 
\begin{equation}\label{eq:LebesgueDensity}
	\dfrac{\leb^{n+1}(\partial \{u>0\}\cap E_r(x,t))}{\leb^{n+1}(E_r(x,t))} \to 1 \quad \text{ as } r\downarrow 0,
\end{equation}
where $E_r (x,t) := B_r(x) \times  (t-r, t+r)$. Let us take one of such points $(x,t)$ (recall that this is allowed by our assumption of $\leb^{n+1} (\partial\{u>0\}\cap Q_R)$ being positive). 
Up to a translation, we may assume that $(x,t) = (0,0)$.
Let us also take $r_\circ>0$ such that, for $r<r_\circ$, $\leb^{n+1}(\partial \{u>0\}\cap E_{r/2}) \geq \leb^{n+1}(E_{r/2}) /2$.

Now, we take a sequence $r_k \to 0$ such that $1/(2r_k)\in \NN$.
For each of these $r_k$ (which from now on we will denote simply by $r$), we decompose the cylinder $E_{r/2}$ (up to a set of zero measure) into $1/r$ disjoint parabolic cylinders $Q_{r/2}(0, t_i)$ for some $\{t_i\}_i$.
If $r< r_\circ$, then  it is not difficult to see that the number of cylinders $Q_{r/2}(0, t_i)$ which intersect $\partial \{u>0\}$, which we denote by $N$, satisfies $N\geq 1/r$.
From such $N$ cylinders, we can take another collection $\{Q_{r/2}(0, t_j)\}_j$ with cardinality at least $N/4$ and such that the distance between each pair of cylinders in the collection is greater than $r^2$.
Thus, we can build another collection of parabolic cylinders $\{Q_{r/2}(z_j, \tau_j)\}_j$, with the same cardinality, such that $(z_j, \tau_j) \in \partial \{u>0\}\cap Q_{r/2}(0, t_j) $.
For each of these cylinders we can use the density property~\eqref{eq:Density} and, adding up (using that the cylinders are pairwise disjoint), we get
\begin{equation}
		\leb^{n+1}(\partial \{u>0\}\cap E_r) \geq \sum_{j= 1}^{N/4} \leb^{n+1}(\partial \{u>0\}\cap Q_{r/2}(z_j, \tau_j)) \geq c_\star \dfrac{N}{4} \leb^{n+1}(Q_{r/2})  \geq \bar{c} \leb^{n+1}(E_r),
\end{equation}
for some constant $\bar{c}>0$ depending only on $n$ and $K_R$.
Note that in the last inequality we have used that $N\geq 1/r$.
As a consequence,
\begin{equation}
	\dfrac{\leb^{n+1}(\partial \{u>0\}\cap E_r) }{\leb^{n+1}(E_r)} \geq \bar{c}  > 0,
\end{equation}
which combined with \eqref{eq:LebesgueDensity} ---recall that we assume $(x,t)=(0,0)$--- yields 
\begin{equation}
	\lim_{r \downarrow 0}\dfrac{\leb^{n+1}(\overline{\{u>0\}}\cap E_r)}{\leb^{n+1}(E_r)}  > 1,
\end{equation}
a contradiction.
\end{proof}
%
%
%
%
%

%
%
%
%
%
%
\section{Extensions}\label{Sec:Extension}
In this final section, we explain how to modify the arguments above to construct strong solutions to
   \begin{equation}\label{eq:BDDDomain1}
	\begin{cases}
		\partial_t u - \Delta u = - f_\gamma(u) \quad &\text{ in }  \Omega \times (0,\infty), \\
		u = 0              \quad &\text{ in }  \partial\Omega \times (0,\infty), \\   
		u|_{t=0} = u_0     \quad &\text{ in } \Omega,
	\end{cases}
	\end{equation}
where $\Omega \subset \RR^n$ is a bounded domain. Before starting with the proof, a couple of remarks:

\medskip

\noindent $\bullet$ The notion of strong solution to \eqref{eq:BDDDomain1} is exactly the one given in Definition \ref{def:StrongSolReacDiffProb}, replacing $\RR^n$ with $\Omega$, $Q$ with $\Omega_\infty := \Omega\times(0,\infty)$, and $H^1(\RR^n)$ with $H^1_0(\Omega)$.

\smallskip

\noindent $\bullet$ As the reader will easily see, the argument below apply with small changes if we impose homogeneous Neumann conditions on the parabolic boundary, instead of the Dirichlet ones (see also \cite{AudritoSerraTilli2021:art}). 
 
\medskip	
	
Now, the idea we follow is basically the same as the one used in the above sections: we consider the family of functionals
\begin{equation}\label{eq:Functional2}
\tilde{\EE}_\vep(w) := \int_0^\infty \frac{e^{- t/\varepsilon}}{\varepsilon} \bigg( \int_\Omega (\vep|\partial_t w|^2 + |\nabla w|^2 ) \dx + 2\int_\Omega w_+^\gamma  \dx \bigg) \dt,
\end{equation}
with $\vep \in (0,1)$, and we seek for minimizers in the space
\[
\tilde{\UU}_0 := \{ u \in \tilde{\UU} : u|_{t=0} = u_0 \text{ and } u(\cdot,t) = 0 \text{ on $\partial \Omega$ \, for a.e. } t > 0\},
\]
where
\[
\tilde{\UU} := \bigcap_{r > 0} H^1(\Omega_r^+), \qquad \text{ with } \ \Omega_r^+ := \Omega \times (0,r^2),
\]
and the initial data satisfies
\begin{equation}\label{eq:DataDirichlet}
\ u_0 \in H_0^1(\Omega) \cap L^\gamma(\Omega) \quad \text{ and } \quad u_0 \geq 0 \quad \text{a.e. in } \Omega.
\end{equation}
Exactly as above, since each element $u \in \tilde{\UU}$ has all (first) weak derivatives in $L^2(\Omega_r^+)$ for every $r > 0$, the equations $u|_{t=0} = u_0$ and $u(\cdot,t) = 0$  on $\partial \Omega$ for a.e. $t > 0$, appearing in the definition of $\tilde{\UU}_0$, must be intended in the sense of traces.

Given these definitions, we may state and prove the following corollary of \Cref{thm:MainIntro}.
\begin{cor}\label{cor:Dirichlet}
Let $n \geq 1$, $\gamma \in [1,2)$ , $u_0$ as in \eqref{eq:DataDirichlet}, and $f_\gamma$ as in \eqref{eq:Reaction}. 
Then there exist a sequence of minimizers $\{u_{\vep_j}\}_{j\in\NN}$ of \eqref{eq:Functional2} in $\tilde{\UU}_0$ and a  strong solution $u \in \tilde{\UU}_0$ to \eqref{eq:BDDDomain1}, continuous in $\Omega_\infty$, such that 
\[
\begin{aligned}
&u_{\vep_j} \rightharpoonup u \quad \text{ weakly in } \tilde{\UU} \\
&u_{\vep_j} \to u \quad \text{ locally uniformly in } \Omega_\infty.
\end{aligned}
\]
\end{cor}
\begin{proof} We summarize the proof in three remarks as follows.

\medskip

\noindent $\bullet$ The methods used to prove Lemma \ref{lem:ExistenceMin} apply in this setting too, with minor changes: the assumptions on $u_0$ allow us to use it as competitor and obtain the estimate \eqref{eq:LevelEst}, while the proof of existence of minimizers (for fixed $\vep \in (0,1)$) is the same, once we replace $B_r$ with $\Omega$ and $Q_r^+$ with $\Omega_r^+$.
  
\smallskip 

\noindent $\bullet$ Also the energy estimates stated in Proposition \ref{thm:UnifEnergyEst} hold for minimizers of $\tilde{\EE}_\vep$ in $\tilde{\UU}_0$ if we replace $\RR^n$ with $\Omega$. Indeed, it is enough to replace $\RR^n$ with $\Omega$ in the definitions of $I$ and $R$, below formula \eqref{eq:ShortJFuncInner}. The rest of the proof is exactly the same.

\smallskip

\noindent $\bullet$ Finally, we notice that the proofs of \Cref{Lem:EulerEqMinFF}, \Cref{Lem:NonDeg}, \Cref{lem:ParNonDeg}, \Cref{lem:Optimalgrowth}, and \Cref{cor:OptimalReg} are purely local and do not depend on the boundary behavior of minimizers. Therefore, minimizers of $\tilde{\EE}_\vep$ in $\tilde{\UU}_0$ satisfy the same statements with $Q$ replaced by $\Omega_\infty$.

\medskip

Putting such remarks together and proceeding as in the proof of Theorem \ref{thm:MainIntro}, our statement follows.
\end{proof}

\bigskip

\section*{Acknowledgements}

We thank the anonymous referee for the careful reading of the first version of this article and for the helpful comments he/she gave us.

%
%
%
%
%
%
%
%
%
%
%
%
%
%
%



\begin{thebibliography}{99}

\bibitem{AkagiStefanelli2016:art}
{\sc G. Akagi, U. Stefanelli.} \emph{A variational principle for gradient flows of nonconvex energies}, J. Convex Anal. \textbf{23} (2016), 53--75.

\bibitem{AltPhilips1986:art}
{\sc H. W. Alt, D. Phillips}. \emph{A free boundary problem for semilinear elliptic equations}, J. Reine Angew. Math. \textbf{368} (1986), 63--107.

\bibitem{Audrito2021:art}
{\sc A. Audrito.} \emph{On the existence and H\"older regularity of solutions to some nonlinear Cauchy-Neumann problems}, ArXiv preprint, arXiv:2107.03308 (2021).

\bibitem{AudKuk2022:art}
{\sc A. Audrito, T. Kukuljan.} \emph{The Stefan problem with fully nonlinear diffusion}, ArXiv preprint, arXiv:2208.14791 (2022).

\bibitem{AudritoSerraTilli2021:art}
{\sc A. Audrito, E. Serra, P. Tilli.} \emph{A minimization procedure to the existence of segregated solutions to parabolic reaction-diffusion systems}, Comm. Partial Differential Equations  \textbf{46} (2021), 2268--2287.

\bibitem{BogeleinEtAl2014:art}
{\sc V. B\"ogelein, F. Duzaar, P. Marcellini.} \emph{Existence of evolutionary variational solutions via the calculus of variations}, J. Differential Equations \textbf{256} (2014), 3912--3942.

\bibitem{BogeleinEtAl2017:art}
{\sc V. B\"ogelein, F. Duzaar, P. Marcellini, S. Signoriello.} \emph{Parabolic equations and the bounded slope condition}, Ann. I. H. Poincar\'e \textbf{34} (2017), 355--379.

\bibitem{Caf78}
{\sc L. Caffarelli}. \textit{The regularity of free boundaries in higher dimensions}, Acta Math. \textbf{139} (1978), 155--184.

\bibitem{Caf78Bis}
{\sc L. Caffarelli}. \textit{Some aspects of the one-phase Stefan problem}, Indiana Math. J. \textbf{27} (1978), 73--77.

\bibitem{CF79}
{\sc L. Caffarelli, A. Friedman}. \textit{Continuity of the temperature in the Stefan problem}, Indiana U. Math. J. \textbf{28} (1979), 53--70.

\bibitem{CPS}
{\sc L. Caffarelli, A. Petrosyan, H. Shahgholian}. \emph{Regularity of a free boundary in parabolic potential theory},
J. Amer. Math. Soc. \textbf{17} (2004), 827--869.

\bibitem{ChoeWeiss2003}
{\sc H. J. Choe, G. S. Weiss.} \emph{A semilinear parabolic equation with free boundary}, Indiana Univ. Math. J. \textbf{52} (2003), 19--50.

\bibitem{DeGiorgi1996:art}
{\sc E. De Giorgi}. \emph{Conjectures concerning some evolution problems}, A celebration of John F. Nash, Jr. Duke Math. J. \textbf{81} (1996), 255--268.

\bibitem{DeSilSav2022:art}
{\sc D. De Silva, O. Savin}. \emph{Uniform density estimates and $\Gamma$-convergence for the Alt-Phillips functional of negative powers}, ArXiv preprint, arXiv:2205.08436.

\bibitem{DeSilSav2022bis:art}
{\sc D. De Silva, O. Savin}. \emph{Compactness estimates for minimizers of the Alt-Phillips functional of negative exponents}, ArXiv preprint, arXiv:2211.00553.

\bibitem{DiKarVal2022:art}
{\sc S. Dipierro, A. Karakhanyan, E. Valdinoci}. \emph{Classification of global solutions of a free boundary problem in the plane}, ArXiv preprint, arXiv:2203.11663.

\bibitem{DurGiac2022:art}
{\sc R. Durastanti, L. Giacomelli}. \emph{Spreading Equilibria Under Mildly Singular Potentials: Pancakes Versus Droplets}, J. Nonlinear Sci. \textbf{32} (2022), 1--61.

\bibitem{FerRERosOt2022:book}
{\sc X. Fern\'andez-Real, X. Ros-Oton}. \emph{Regularity Theory for Elliptic PDE}, Zurich Lectures in Advanced Mathematics, EMS books, 2022.

\bibitem{FigRosSerra1}
{\sc A. Figalli, X. Ros-Oton, J. Serra}. \textit{The singular set in the Stefan problem}, ArXiv preprint, arXiv:2103.13379 (2021).

\bibitem{FriedKinder1975:art}
{\sc A. Friedman, D. Kinderlehrer.} \emph{A one phase Stefan problem}, Indiana Univ. Math. J. \textbf{24} (1975), 1005--1035.

\bibitem{Ilmanen1994:art}
{\sc T. Ilmanen.} \emph{Elliptic regularization and partial regularity for motion by mean curvature}, Mem. Amer. Math. Soc. \textbf{108} (1994), no. 520, x+90 pp.

\bibitem{Lieberman1996:book}          
{\sc G.-M. Lieberman}. \emph{Second order parabolic differential equations}, World Scientific Publishing Co., 1996.

\bibitem{Lions1965:art}          
{\sc J.-L. Lions}. \emph{Sur certaines \'equations paraboliques non lin\'eaires}, Bull. Soc. Math. France, \textbf{93} (1965), 155--175.

\bibitem{Martinson1976:art}
{\sc L. K. Martinson}. \emph{The finite velocity of propagation of thermal perturbations in media with constant thermal conductivity}, U.S.S.R. Comput. Math. and Math. Phys. \textbf{16} (1976), 141--149.

\bibitem{MielkeStefanelli2011:art}
{\sc A. Mielke, U. Stefanelli}. \emph{Weighted energy-dissipation functionals for gradient flows}, ESAIM Control Optim. Calc. Var. \textbf{17} (2011), 52--85.

\bibitem{Oleinik1964:art}
{\sc  O. A. Oleinik}. \emph{On a problem of G. Fichera}, Dokl. Akad. Nauk SSSR \textbf{157} (1964), 1297--1300.

\bibitem{Phillips1983:art}
{\sc D. Phillips}. \emph{Hausdoff measure estimates of a free boundary for a minimum problem}, Comm. Partial Differential Equations, \textbf{8} (1983), 1409--1454.

\bibitem{Phillips1983bis:art}
{\sc D. Phillips}. \emph{A minimization problem and the regularity of solutions in the presence of a free boundary}, Indiana University Mathematics Journal, \textbf{32} (1983), 1--17.

\bibitem{Phillips1987:art}
{\sc D. Phillips}. \emph{Existence of solutions of quenching problems}, Appl. Anal., \textbf{24} (1987), 253--264.


\bibitem{RSSS:2019}
{\sc R. Rossi, G. Savar\'e, A. Segatti, U. Stefanelli.} \emph{Weighted energy-dissipation principle for gradient flows in metric spaces}, J. Math. Pures Appl. \textbf{127} (2019), 1--66.

\bibitem{SerraTilli2012:art}
{\sc E. Serra, P. Tilli}. \emph{Nonlinear wave equations as limits of convex minimization problems: proof of a conjecture by De Giorgi}, Ann. of Math. \textbf{175} (2012), 1551--1574.

\bibitem{SerraTilli2016:art}
{\sc E. Serra, P. Tilli}. \emph{A minimization approach to hyperbolic Cauchy problems}, J. Eur. Math. Soc. \textbf{18} (2016), 2019--2044.

\bibitem{Simon1987:art}
{\sc J. Simon}. \emph{Compact sets in the space $L^p(0,T;B)$}, Ann. Mat. Pura Appl. \textbf{146} (1987), 65--96.

\bibitem{Wang92bis} 
{\sc L. Wang}. \textit{On the regularity theory of fully nonlinear parabolic equations: II}, Comm. Pure Appl. Math. \textbf{45} (1992), 141--178.

\bibitem{Weiss1999:art}
{\sc G. S. Weiss.} \emph{Self-similar blow-up and Hausdorff dimension estimates for a class of parabolic free boundary problems}, SIAM J. Math. Anal. \textbf{30} (1999), 623--644.

\bibitem{Weiss2000:art}
{\sc G. S. Weiss.} \emph{The free boundary of a thermal wave in a strongly absorbing medium}, J. Differential Equations \textbf{160} (2000), 357--388. 

\end{thebibliography}
\end{document}